\documentclass[12pt,a4paper]{amsart}

\usepackage{amsmath,amssymb,amsthm,calc}

\usepackage[all]{xy}
\usepackage[top=35mm, bottom=25mm, left=25mm, right=25mm]{geometry}

\usepackage[pdftex]{graphicx}

\usepackage{tikz}

\newtheorem{theorem}{Theorem}[section]
\newtheorem{proposition}[theorem]{Proposition}
\newtheorem{lemma}[theorem]{Lemma}
\newtheorem{corollary}[theorem]{Corollary}

\theoremstyle{definition}
\newtheorem{definition}[theorem]{Definition}
\newtheorem{remark}[theorem]{Remark}

\newtheorem{example}[theorem]{Example}

\newcommand{\im}{\operatorname{Im}}

\newcommand{\height}{\operatorname{height}}

\newcommand{\spec}{\operatorname{Spec}}

\newcommand{\V}{\operatorname{V}}
\newcommand{\supp}{\operatorname{Supp}}

\newcommand{\Mon}{\operatorname{Mon}}

\newcommand{\Hom}{\operatorname{Hom}}

\newcommand{\R}{\operatorname{R}}

\newcommand{\Der}{\operatorname{Der}}

\DeclareMathOperator{\link}{lk}

\newcommand{\bfa}{\mathbf{a}}
\newcommand{\bfb}{\mathbf b}

\newcommand{\kk}{\Bbbk}
\newcommand{\ZZ}{\mathbb{Z}}
\newcommand{\NN}{\mathbb{N}}

\newcommand{\HomI}{\Hom_{R}(I,R/I)}

\newcommand{\Nb}{\overline{N}}
\newcommand{\HomK}{\Hom_{R}(K/K_0,R/I)}

\begin{document}

\title{Cotangent cohomology of quadratic monomial ideals}
\author{Amin Nematbakhsh}
\address{School of Mathematics, Institute for Research in Fundamental Sciences (IPM), P. O. Box: 19395-
5746, Tehran, Iran}
\email{nematbakhsh@ipm.ir}
\date{\today}

\begin{abstract} We study the deformation theory of quotients of polynomial rings by quadratic monomial ideals. More precisely we compute the first cotangent cohomology module of such rings. We also give a criterion for vanishing of second cotangent cohomology module.
\end{abstract}

\maketitle

\section*{Introduction}

In deformation theory of affine schemes there is a cohomology theory which assigns to any $\kk$-algebra $A$ two cohomology modules called the first and second cotangent cohomology modules denoted by $T^1(A)$ and $T^2(A)$. We refer to Section \ref{FirstOrder} for definitions. The first cotangent cohomology module characterizes the first order deformations of $A$ and the second cotangent cohomology module contains the obstructions for lifting these deformations.
In this paper we investigate cotangent cohomology of quotients of polynomial rings by quadratic monomial ideals.

Any ideal in a polynomial ring with a Gr\"obner basis consisting of quadrics degenerate to a quadratic monomial ideal.
Such ideals include Hibi ideals and Pl\"ucker relations of Grassmann varieties.

A quadratic monomial ideal $I$ in a polynomial ring $R=\kk[x_1,\ldots,x_n]$ gives rise to a (not necessarily simple) graph $G=(V(G),E(G))$ where $V(G)=\{x_1,\ldots,x_n\}$ and $E(G)=\{\{x_i,x_j\}~|~ x_ix_j \in I\}$.
We use the combinatorics of the corresponding graph to describe a generating set for the first cotangent cohomology module of the ring $R/I$ as well as vanishing results for the second cotangent cohomology module.


\medskip
{\it Cotangent cohomology of Stanley-Reisner rings.}
Deformation theory of square-free monomial ideals have been studied by Klaus Altmann and Jan Arthur Christophersen in \cite{AltCh-Cotang,AltCh-DefoSR}.
If $I$ is a monomial ideal in a polynomial ring $R=\kk[x_1,\ldots,x_n]$ then 
$T^1(R/I)$ is $\ZZ^n$-graded.
When in addition $I$ is a square-free monomial ideal then there is a unique simplicial complex $\Delta$ on $n$ vertices, such that $I$ is the Stanley-Reisner ideal of $\Delta$.
For a subset $ g \subseteq [n]$, the {\it link} of $g$ in $\Delta$ is defined to be
\[
\link_{\Delta} G = \{f \in \Delta~|~ f \cap g = \emptyset, f \cup g \in \Delta\}.
\]
In \cite{AltCh-Cotang,AltCh-DefoSR} the authors give a combinatorial description of each $\ZZ^n$-graded part of $T^1(R/I)$.
More precisely, let $\mathbf{c} \in \ZZ^n$ be a multidegree and suppose $\mathbf{c}=\mathbf{a}-\mathbf{b}$ with $\mathbf{a},\mathbf{b}\in \NN$ and $\supp \bfa \cap \supp \bfb=\emptyset$. Recall that for a multidegree $\bfa =(a_1,\ldots,a_n) \in \ZZ^n$, $\supp \bfa =\{i \in [n] | a_i\neq 0\}$. We have
\begin{enumerate}
\item if $\mathbf{b}\notin \{0,1\}^n$ then $T^1(R/I)_{\mathbf{a}-\mathbf{b}} =0$;
\item if $\mathbf{b}\in \{0,1\}^n$ then $T^1(R/I)_{\mathbf{a}-\mathbf{b}} = T^1(\link_\Delta \supp \mathbf a)_{-\mathbf{b}}$.
\end{enumerate}
This description of the first cotangent cohomology is essentially used in \cite{ABHL} to classify rigid square-free monomial ideals and in particular find the class of rigid edge ideals of graphs. Due to simplicity of relations of a quadratic monomial ideal it is possible to apply a more direct approach towards description of the first cotangent cohomology module. Here we construct a homogeneous generating set for $T^1(R/I)$ as a $\ZZ$-graded module. This generating set is easier to use when examining the rigidity of a quadratic monomial ideal.

\medskip
{\it Rigidity.}
A $\kk$-algebra $A$ is called {\it rigid} if all infinitesimal deformations of $A$ are trivial, i.e $T^1(A)=0$.
Let $G$ be a simple graph on vertex set $x_1,\ldots,x_n$ and let $R=\kk[x_1,\ldots,x_n]$ be a polynomial ring on variables $x_i$.
The {\it edge ideal} of $G$ denoted by $I(G)$ is the ideal in $R$ generated by quadratic monomials $x_ix_j$ such that $\{x_i,x_j\}$ is an edge of $G$.
An {\it independent set} $A$ of a graph $G$ is a subset of $V(G)$ for which no two vertices of an edge of $G$ belong to $A$.
For a vertex $x$ of $G$ let $N(x) = \{y \in V(G)| \{x,y\} \in E(G)\}$ be the {\it neighborhood} of $x$.
We denote by $\overline{N}(x)$ the complementary graph of the induced subgraph of $G$ on vertex set $N(x)$.
The {\it neighborhood} of a set $X$ of vertices of $G$ is defined to be $N(X)=\cup_{x\in X} N(x)$, and the {\it closed neighborhood} of $X$ is defined to be $N[X] = X\cup N(X)$. We also denote the induced subgraph of $G$ on the vertex set $V(G)\backslash X$ by $G\backslash X$.
In \cite{ABHL}, it is shown that $R/I(G)$ is rigid if and only if any independent subset $X$ of $G$ satisfies both of the following conditions.
\begin{enumerate}
\item $\overline{N}(x)$ is connected for all vertex $x$ of graph $G\backslash N[X]$;
\item $G\backslash N[X]$ contains no isolated edge.
\end{enumerate}
In Theorem \ref{thm-mainrigidity}, we give another characterization for rigidity of quadratic monomial ideals. This characterization examines the rigidity of the edge ideal of a graph $G$ in small neighborhoods of vertices and edges of $G$.


\medskip
{\it Inseparable graphs}. We call a graph $G$ {\it inseparable} if its edge ideal $I(G)$ is inseparable (see Section \ref{Separation} for definitions). A combinatorial characterization of inseparable simple graphs is given in \cite[Theorem 3.1]{ABHL}. Separation of edge ideal of a (not necessarily simple) graph is again a quadratic monomial ideal. Let $J$ be a separation of edge ideal $I(G)$ of a graph $G$. In Section \ref{Separation} we give a construction for a graph $H$ for which its edge ideal gives the ideal $J$. This approach towards a characterization of inseparable graphs is in spirit the same as what the authors did in \cite[Section 3]{ABHL}. 

\medskip
{\it Organization of paper.} In Section \ref{Sec-defo}, we first recall some notions on graphs and their edge ideals. In Sections \ref{Polarization} and \ref{Separation}, we define polarization and separations of monomial ideals and describe such operations in case of edge ideals of graphs.
Polarization and separations of a monomial ideal are special cases of deformations. We investigate whether these ideals posses other deformations.
In Section \ref{FirstOrder}, we also provide preliminaries on the deformation theory of rings.

Section \ref{T1} contains the computation of the first cotangent cohomology module. We give a generating set for the first cotangent cohomology of all quadratic monomial ideals.
More precisely, we describe a set of unobstructed first order deformations that generate $T^1(R/I)$ as an $R/I$-module. These deformations correspond to $\ZZ$-graded homomorphism in $\HomI$.
It is quite interesting that all the separations of the ideal $I$ appear among these generating elements of $T^1(R/I)$.
Having such deformations in hand it is easy to compute a basis for each $\ZZ$-graded component of $\HomI$.
This also enables us to give another characterization of rigid edge ideals of graphs (see Theorem \ref{thm-mainrigidity}).

Section \ref{Rigidity} investigates the rigidity of quadratic monomial ideals. We provide a characterization for rigidity of such ideals.

In Section \ref{T2} we give a criterion for vanishing of the second cotangent cohomology module.
The second cotangent cohomology module is an obstruction space and only its vanishing is of importance.
Characterization of quadratic monomial ideals for which the second cotangent cohomology vanishes seems to be difficult in general. However when $G$ is a simple graph with no 3-cycles then there is a nice characterization for vanishing of the second cotangent cohomology module (see Theorem \ref{thmMainT2}).
We also show that if the graph $G$ does not have any induced 3 or 4 cycles then the second cotangent cohomology module vanishes.


{\it Acknowledgement.} The author would like to thank Gunnar Fl{\o}ystad for his valuable comments and suggestions. 

\section{Deformations of edge ideals of graphs}
\label{Sec-defo}

Let $G=(V(G),E(G))$ be a finite graph. Let $R=\kk[G]$ be a polynomial ring with vertices of graph $G$ as indeterminates.
The {\it edge ideal} $I(G)$ of $G$ is a quadratic monomial ideal in $R$ generated by monomials $ab$ where $\{a,b\}$ is an edge of $G$.
Throughout $ab$ denotes both an edge $\{a,b\}$ of $G$ and also the monomial assigned to this edge in $I(G)$.
We define the {\it underlying simple graph} of a graph $G$ simply as a graph we obtain from $G$ by removing all loops and substituting multiple edges with only one edge.
The edge ideal of a graph is usually defined for simple graphs but we do not need our graphs to be simple until the last section.
Since multiple edges do not change the edge ideal we may assume that $G$ has no multiple edges. Furthermore since isolated vertices only change the ambient polynomial ring and they do not change the edge ideal we also assume that $G$ has no isolated vertices.

If $G$ is a simple graph then $I(G)$ is a square-free monomial ideal and it coincides with the Stanley-Reisner ideal of the clique complex of the complementary graph of $G$ (see Section 1.5 in \cite{HeHiMon}).


Let $e$ be an edge of $G$. We call an edge $e'$ distinct from $e$ an {\it adjacent edge} to $e$ when $e$ and $e'$ have a vertex in common.
Recall that a {\it leaf vertex} or a {\it free vertex} is a vertex of degree 1. Following \cite{ABHL}, an edge of graph $G$ is called a {\it leaf} if it has a leaf vertex and it is called a {\it branch} if it is connected to a leaf other than itself. 
For distinct vertices $a$ and $b$ of $G$, we call an edge $ab$ of $G$ an {\it isolated edge} when $a$ and $b$ are leaf vertices. We also call a loop on vertex $a$ an {\it isolated loop} provided that $a$ is a vertex of degree 2.
For a subset $H$ of the vertex set of a graph $G$ the {\it induced subgraph} of $G$ on $H$ is defined to be the subgraph of $G$ on vertex set $H\subseteq V(G)$ and edge set consisting of exactly those edges of $G$ connecting pairs of vertices in $H$.

For a monomial $m$ in the polynomial ring $R$ we denote the largest square-free monomial that divides $m$ by $\sqrt{m}$.
For a polynomial $f \in R$ we denote the set of monomials that appear with nonzero coefficients in $f$ with $\Mon(f)$. For example if $f=2x^2y - 3xyz$ then $\Mon(f) = \{x^2y, xyz\}$ and $\sqrt{x^2y} = xy$.

\subsection{Polarization of edge ideals of graphs}
\label{Polarization}

Let $I$ be a monomial ideal in a polynomial ring $R=\kk[X]$ where $X=\{x_1,\ldots,x_n\}$ is a set of indeterminates.
The polarization of $I$ is a square-free monomial ideal assigned to $I$ in a larger polynomial ring as follows.
Let $G(I)=\{g_1,\ldots,g_r\}$ be the set of minimal generators of $I$. 
For each $i=1,\ldots,n$ let $e_i$ be the highest power of $x_i$ among the elements of $G(I)$.
We define a new polynomial ring $S=R[y_{i,j}|i=1,\ldots,n;j=2,\ldots,e_i]$.
Now for $g\in G(I)$ if $g=x_1^{a_1} \cdots x_n^{a_n}$ then define
\[
\tilde{g} = x_1^{\min\{a_1,1\}} y_{1,2}\cdots y_{1,a_1} x_2^{\min\{a_2,1\}} y_{2,2}\cdots y_{2,a_2} \cdots x_n^{\min\{a_n,1\}} y_{n,2}\cdots y_{n,a_n}.
\]

We call the ideal $J$ defined by $\tilde{g_1},\ldots,\tilde{g_n}$ in $S$ the {\it polarization of $I$}.


Let $G=(V(G),E(G))$ be a graph and $I(G)$ be its edge ideal. If $G$ has loops on vertices $x_1,\ldots,x_m$ then define a new graph $\tilde{G}$ over the vertex set $V(\tilde{G})=V(G)\cup\{y_1,\ldots,y_m\}$ with edge set
\[E(\tilde{G})= \big( E(G)- \cup_{i=1}^m \big\{ \{x_i,x_i\} \big\} \big) \bigcup \big( \cup_{i=1}^m \big\{ \{x_i,y_i\} \big\} \big).\]
It follows that $I(\tilde{G})$ is the polarization of $I(G)$.
Therefore in case of edge ideal of graphs the polarization is the edge ideal of a graph constructed by removing all the loops and adding leaves instead.

Polarization is a special case of a more general process called separation.

\subsection{Separations of edge ideals of graphs}
\label{Separation}

Let $I$ be a monomial ideal in the polynomial ring $R=\kk[X]$. Let $x\in X$ be an indeterminate of $R$ and let $y$ be an indeterminate over $R$.
Following the definition in \cite{FGH} a monomial ideal $J$ in $S=R[y]$ is a called a {\it separation} of $I$ at the variable $x$ if 
\begin{enumerate}
\item $I$ is the image of $J$ under the $\kk$-algebra map $S \to R$ sending $y$ to $x$ and any other variable of $S$ to itself,
\item $x$ and $y$ occur in some minimal generators of $J$ and
\item $y-x$ is a regular element of the quotient ring $S/J$.
\end{enumerate}
We shall call a succession of separations also a {\it separation}.
We call $x$ a {\it separating variable}.
The ideal $I$ is called {\it separable} if it admits a separation, otherwise it is called {\it inseparable}.

For a vertex $x$ let $N(x) = \{y~|~\{x,y\} \in E(G)\}$ be the neighborhood of $x$.
Note that $x\in N(x)$ if and only if $G$ has a loop on $x$.

\begin{definition}
Let $G$ be a graph. We call a vertex $v$ a {\it separating vertex} when either of the following conditions hold
\begin{enumerate}
\item the neighborhood of $v$ can be divided into two nonempty disjoint subsets $A,B$ such that any two vertices of $A$ and $B$ are adjacent. We call the pair $(A,B)$ a {\it separation pair} of $v$;
\item there is an isolated loop on vertex $v$. In this case $(\emptyset,\{v\})$ is called a {\it separation pair} of $v$.
\end{enumerate}
\end{definition}

Let $G$ be a graph with a separating vertex $v$. We construct another graph from $G$ in the following way. First we add a new vertex $v'$ to $G$. We also remove any edge between vertices in $B$ and $v$ and then we connect $v'$ to any vertex in $B$. We call the new graph $H$, a {\it separation} of $G$ at vertex $v$.

\medskip
It is not hard to show that any separation pair gives a separation of the edge ideal of $G$.
Let $v$ be a separating vertex of $G$ with separation pair $(A,B)$.
Let $R=\kk[G]$ and $S=R[v']$.
There is an algebra homomorphism $S \to R$, sending $v,v'$ to $v$ and any other variable to itself.
Let $H$ be the separation of $G$ with respect to the vertex $v$ and separation pair $(A,B)$. Evidently, the edge ideal $J$ of $H$ maps onto the edge ideal $I$ of $G$ under the map $S \to R$ above. The condition $(2)$ of a separation is also satisfied. Suppose $f(v-v')=0$ in $S/J$ for a polynomial $f$. By \cite[Lemma 7.1]{FGH}, for any monomial $m$ in $f$ we have $mv=mv'=0$ in $S/J$. This shows that either $m\in J$ or for some minimal generators $g_1$ and $g_2$ of $J$ and polynomials $f_1$ and $f_2$, $mv=f_1g_1$ and $mv'=f_2g_2$. It follows that $g_1=va$ and $g_2=v'b$ for vertices $a \in A$ and $b \in B$. But this shows that $ab$ divides $m$ and since $(A,B)$ is a separation pair, $m$ belongs to $J$. Hence $f$ belongs to $J$.

\medskip
The following proposition shows that every separation of an edge ideal is constructed this way. Informally, the separation of an edge ideal is the same as the edge ideal of the separation graph.

\begin{proposition}
Let $G$ be a graph. The edge ideal $I$ of $G$ has a separation at the variable $v$ if and only if $v$ is a separating vertex.
\end{proposition}

\begin{proof}
The construction above shows that a separating vertex $v$ gives a separation at the variable $v$. Conversely, Suppose $I$ has a separation $J$ at variable $v$. Let $H$ be the graph of $J$. Suppose $v,v'$ are the two variables that map to $v$.
Let $A=N(v)\backslash \{v'\}$ and $B=N(v')$.
$A$ and $B$ are disjoint, otherwise for $a \in A\cap B$, $J$ contains minimal generators $av,av'$ which can not happen since $v-v'$ is regular over $J$.
Note that $A\cup B$ equals the neighborhood of $v$ in $G$. If elements $a \in A$ and $b\in B$ are not adjacent in $G$ then they are neither adjacent in $H$. But then $ab(v-v')$ belongs to $J$ contradicting the fact that $v-v'$ is regular. Thus $(A,B)$ is a separation pair of vertex $v$. Furthermore it is not hard to show that the separation graph that we get at the separating vertex $v$ and separation pair $(A,B)$ by the construction above, actually gives the same separation ideal $J$.  
\end{proof}

Let $N(x)$ be the neighborhood of $x$.
Let $G_{N(x)}$ be the induced subgraph of $G$ on the vertex set $N(x)$.
At any vertex $x$ we define a simple graph $\overline{N}(x)$ as the complementary graph of the underlying simple graph of $G_{N(x)}$. If $G$ is not a simple graph then $G$ is obviously separable.
If $G$ is a simple graph then the vertex $x$ of $G$ is a separating vertex if and only if $\overline{N}(x)$ is disconnected.
It follows that a graph $G$ is inseparable if and only if it is simple and for all $x\in V(G)$, $\Nb(x)$ is connected
(see \cite[Theorem 3.1]{ABHL}).

\medskip
Let $I$ be an ideal in a polynomial ring $R$. Let $J \in R[y]$ be a separation of $I$ at variable $x$.
We apply the coordinate change $y \rightsquigarrow x+t$ and we get an ideal $\tilde{I} \in R[t]$ such that $R[y]/J \cong R[t]/\tilde{I}$.
Since $y$ is a nonzero divisor on $R[y]/J$, $t$ is a nonzero divisor on $R[t]/\tilde{I}$. Hence $R[t]/\tilde{I}$ is flat over $\kk[t]$.
Furthermore,
\[
\frac{R[t]}{\tilde{I}} \otimes_{\kk[t]} \frac{\kk[t]}{(t)} 
= \frac{R[t]}{\tilde{I}} \otimes_{\R[t]} R[t] \otimes_{\kk[t]} \frac{\kk[t]}{(t)}
= \frac{R[t]}{\tilde{I}} \otimes_{\R[t]} \frac{R[t]}{(t)} = \frac{R}{I}
\]
This means that any separation $J$ of $I$ at a variable $x$ is a flat deformation of $I$ over the polynomial ring $\kk[t]$.

\subsection{First order deformations}
\label{FirstOrder}

General references for deformation theory are \cite{DefoHar} and \cite{Ser}.

Let $I$ be an ideal in a $\kk$-algebra $R$. Let $B$ be another $\kk$-algebra with a distinguished $\kk$-point $b\in \spec B$ corresponding to a morphism $B\to \kk$. A {\it deformation} of $I$ over $B$ is an ideal $J$ in $R\otimes_\kk B$ satisfying the following
\begin{enumerate}
\item $(R\otimes_\kk B)/J$ is flat over $B$,
\item the natural map $R\otimes_\kk B \to R$ induces an isomorphism $(R\otimes_\kk B)/J \otimes_B \kk \to R/I$.
\end{enumerate}

If $B$ is a local Artinian $\kk$-algebra such that $B/m_B \cong \kk$ then a deformation over $B$ is called an {\it infinitesimal deformation}. A deformation over the local Artinian ring $\kk[\epsilon] = \kk[t]/(t^2)$ is call a {\it first order deformation} of $R/I$.
Suppose $J\subseteq R[\epsilon]$ is an ideal such that $R[\epsilon]/J \otimes_{\kk[\epsilon]} \kk[t]/(t^2) \cong R/I$. If $I = (f_1,\ldots,f_r)$ then $J=(f_1+g_1\epsilon,\ldots,f_r+g_r\epsilon)$ and $R[\epsilon]/J$ is flat over $\kk[\epsilon]$ if and only if the map sending $f_i\mapsto g_i + I$ defines a well-defined $R$-module homomorphism $I\to R/I$. Therefore the set of first order deformations of $R/I$ are in one-to-one correspondence with elements of $\Hom_R(I,R/I)$.

\begin{remark}
Let $I$ be an ideal in a polynomial ring $R$ and let $J$ be a separation of $I$ at a variable $x$.
Suppose $I=(f_1,\ldots,f_r)$ and $J=(g_1,\ldots,g_r)$ such that under the map $R[y] \to R$, $g_i$ maps to $f_i$ for $i=1,\ldots,r$.
Suppose $y$ divides $g_1,\ldots,g_k$ and no other generator of $J$ is divisible by $y$.
The assignment $f_i \mapsto f_i/x$ for $i=1,\ldots,k$ and $f_i\mapsto 0$ for $i=k+1,\ldots,r$ defines a homomorphism in $\HomI$ which corresponds to the deformation
\[
(f_1+(f_1/x) t,\ldots,f_k + (f_k/x) t, f_{k+1},\ldots,f_r)
\]
of $I$.
Note that if we substitute $t$ by $y-x$ we get the ideal $J$.
\end{remark}

\subsection{Cotangent cohomology}
Let $I=(f_1,\ldots,f_r)$ be an ideal in a polynomial ring $R$ and let $A= R/I$.
Let $\Der_\kk(R)$ be the module of derivations of $R$. If $R=\kk[x_1,\ldots,x_n]$ then $\Der_\kk(R)$ is a free $R$-module generated by derivations $\frac{\partial}{\partial x_i}$ for $i=1,\ldots,n$. There is a map
\[
\delta^\ast : \Der_\kk(R) \longrightarrow \Hom_R(I,R/I)
\]
which sends $\partial$ to the homomorphism sending $f_i\mapsto \partial f_i+I$ for $i=1\ldots,r$.
We usually denote the image of a derivation $\partial$ under the map $\delta^\ast$ again by $\partial$. 
The cokernel of the map $\delta^\ast$ is called the {\it first cotangent cohomology module} of $A$ and it is denoted by $T^1(A)$.
A homomorphisms in $\HomI$ is called a {\it trivial first order} deformation if it lies in the image of $\delta^\ast$ and it is called a {\it nontrivial first order} deformation otherwise.
Therefore $T^1(A)$ characterizes all the nontrivial first order deformations of $A$.
A ring $A=R/I$ as well as the ideal $I$ is called {\it rigid} if $T^1(A)$ vanishes.
Following \cite{FGH} we call a graph $G$ {\it algebraically rigid} (or simply {\it rigid}) if its edge ideal $I(G)$ is rigid.

Now let
\[
0 \longrightarrow K \longrightarrow R^m \stackrel{j}{\longrightarrow} R \longrightarrow A \longrightarrow 0
\]
be an exact sequence presenting $A$ as an $R$-module.
Let $\epsilon_1,\ldots,\epsilon_m$ be a basis for $R^m$ and let $K_0$ be the submodule of $K$ generated by relations $j(\epsilon_i)\epsilon_j - j(\epsilon_j)\epsilon_i$ for all $i\neq j$, $1\leq i,j \leq m$. These relations are called the Koszul relations.
Note that $K/K_0$ is an $A$-module.  The cokernel of the map
\[
\Phi: \Hom_R(R^m,A) \longrightarrow \Hom_A(K/K_0,A)
\]
is called the {\it second cotangent cohomology module} of $A$ and is denoted by $T^2(A)$. 

The modules $T^i(A)$ for $i=1,2$, are originally defined as cohomology of
$\Hom_A(L_\bullet, A)$ where $L_\bullet$ is a 3-term complex of $A$-modules called the cotangent complex. This definition of cotangent cohomology modules as cohomology of a complex is equivalent to the definition we gave above.

\section{First cotangent cohomology}
\label{T1}

Let $I$ be a monomial ideal generated in degree $d$ in a polynomial ring $R$.
Any $R$-linear map $\phi:I \to R/I$ gives an element $\varphi$ of $\Hom_\kk(I_d,R/I)$ by definition. Conversely, any $\kk$-linear map $\varphi:I_d \to R/I$ which satisfies the relations of $I$ algebraically extends to a well-defined $R$-linear map $\phi:I\to R/I$. Therefore there is a one-to-one correspondence
$$\Hom_R(I,R/I)) \cong \{\phi\in \Hom_\kk(I_d,R/I)| \phi \text{ satisfies the relations of }I\}.$$

Let $I$ be the edge ideal of a graph $G$. Let $R=\kk[G]$.
Two edges $ab$ and $ab'$ with a vertex in common define a relation $b'(ab) - b(ab')$ of $I$. Moreover if two distinct edges $ab$ and $a'b'$ do not have a common vertex then we have a Koszul relation $a'b'(ab)-ab(a'b')$ of $I$ and these relations generate all of the relations of the ideal $I$. 

\medskip
We define two types of homomorphisms in $\HomI$ that define a generating set for $\HomI$ as well as a generating set for $T^1(R/I)$.

\medskip
{\bf Type I}.
Let $ab$ be an edge of $G$. The vertices $a$ and $b$ of edge $ab$ are not necessarily distinct.
Let $\Lambda^{ab}$ be the set of all vertices of $G$ that are adjacent to $a$ or $b$.
More precisely, let $\Lambda^{ab} = (N(a)\backslash \{b\}) \cup (N(b)\backslash \{a\})$. When the edge $ab$ is fixed we usually denote $\Lambda^{ab}$ by $\Lambda$.
We also assume that $a$ (resp. $b$) belongs to $\Lambda$ if and only if there is a loop on $a$ (resp. $b$).

For any $g \in \Lambda$ let $\Lambda_g$ be the set all vertices adjacent to $g$ other than $a$ and $b$, i.e. $\Lambda_g = N(g)\backslash \{a,b\}$.
Let $|\Lambda| = d$. Any ordered $d$-tuple $(x_1,\ldots,x_d)$ in $\prod_{g\in \Lambda} \Lambda_g$ gives a monomial $x_1\cdots x_d$.
By abuse of notation, from now on $\prod_{g\in \Lambda} \Lambda_g$ denotes the set of such monomials instead of the $d$-tuples $(x_1,\ldots,x_d)$.

Now define $\Lambda_{ab}$ as 
\[\Lambda_{ab} = \{ \sqrt{m}~|~ m \in \prod_{g\in \Lambda} \Lambda_g\}.\]

If $\Lambda=\emptyset$ that is when $ab$ is an isolated edge or an isolated loop then $\Lambda_{ab} = \{1\}$. It is worth mentioning that for any $\lambda \in \Lambda_{ab}$ and any $g\in \Lambda$, there exists a vertex $x$ adjacent to $g$ such that $x|\lambda$.

Now for $\lambda \in \Lambda_{ab}$, we define a linear map
\[
\phi_{ab}^\lambda: I_2 \to R/I
\]
which sends $ab$ to $\lambda$ and any other minimal generator of $I$ to zero.

\begin{lemma} \label{lemTypeI}
The map $\phi_{ab}^\lambda$ algebraically extends to a well-defined homomorphism in $\HomI$. Furthermore, if $\phi_{ab}^\lambda$ is nonzero then it corresponds to a nontrivial deformation.
\end{lemma}

\begin{proof}
We show that $\phi=\phi_{ab}^\lambda$ satisfies the relations of $I$. Let $a'b'$ be a generator of $I$ with no common vertex with edge $ab$. It is trivial that the map $\phi$ satisfies the relation $a'b'(ab) - ab(a'b')$, i.e. $a'b' \phi(ab) - ab \phi(a'b') =0$ in $R/I$.
Let $ab'$ be another generator of $I$. By definition of $\Lambda_{ab}$, there exists some $x$ adjacent to $b'$, such that $x|\lambda$. Therefore $\phi$ satisfies the relation $b'(ab) - b(ab')$, since $b'\phi(ab) - b\phi(ab') = b' \lambda = 0$ in $R/I$.
The argument in the case of a generator of form $a'b$ and relation $a'(ab)-a(a'b)$ is similar.

Since $a$ or $b$ does not divide $\lambda$, the homomorphism $\phi$ can not lie in the submodule of $\HomI$ generated by derivations. Hence it corresponds to a nontrivial first order deformation.
\end{proof}

\begin{definition}
For any $\lambda\in \Lambda_{ab}$, we call $\phi_{ab}^\lambda$ a {\it type I deformation associated with the edge $ab$}.
When there is no confusion we denote $\phi^\lambda_{ab}$ simply by $ab\mapsto \lambda$.
\end{definition}

\medskip
{\bf Type II.}
Let $a\in V(G)$ be a vertex.
Let $N(a)$ be the neighborhood of $a$. We denote the complementary graph of the underlying simple graph of $G_{N(a)}$ by $\overline{N}(a)$.
Let $L$ be a nonempty subset of the vertex set of $\Nb(a)$. We usually denote the induced subgraph of $\Nb(a)$ on the vertex set $L$ again by $L$.
Let $\Gamma(L)$ be the set of all vertices in $\Nb(a)$ which are adjacent to some vertex of $L$ but does not belong to $L$. When the subgraph $L$ is fixed we simply denote $\Gamma(L)$ by $\Gamma$.
For any $g \in \Gamma$, let $\Gamma_g$ be the set of vertices adjacent to $g$ other than $a$.
Let
\[ \Gamma_{a,L} = \{\sqrt{m} ~|~ m \in
\prod_{g\in \Gamma} \Gamma_g\}.
\]
Now define a linear map 
\begin{equation*}
\phi_{a,L}^\lambda  : I_2 \to R/I
\end{equation*}
by
\begin{equation*}
\phi_{a,L}^\lambda(e)  = 
\begin{cases}
\lambda x & e=ax \text{ and } x\in L\\
0 & \text{otherwise}.
\end{cases}
\end{equation*}

\begin{lemma} \label{lemTypeII}
The map $\phi_{a,L}^\lambda$ algebraically extends to a well-defined homomorphism in $\HomI$.
\end{lemma}

\begin{proof}
We show that $\phi=\phi^\lambda_{a,L}$ satisfies the relations of $I$.
Obviously, $\phi$ satisfies the Koszul relations.
For $x\in L$, let $ax$ and $bx$ be two generators of $I$.
The relation $b(ax)-a(bx)$ implies $b\phi(ax)-a\phi(bx) = (bx)\lambda=0$ in $R/I$.
For $ax,ax'$ and the relation $x'(ax)-x(ax')$ we have
\begin{enumerate}
\item if $x,x'\in L$ then $x'\phi(ax)-x\phi(ax') = x'\lambda x - x \lambda x' = 0$;
\item if $x\in L$ and $x'\in \Gamma$ then $x'\phi(ax)-x\phi(ax') = x'\lambda x - 0\in I$ by definition of $\Gamma_{a,L}$;
\item if $x\in L$ and $x' \notin (L\cup\Gamma(L))$ then $x'\phi(ax)-x\phi(ax')= x' \lambda x \in I$ since $xx'$ is a generator of $I$.
\end{enumerate}
This completes the proof that $\phi$ is a well-defined $R$-linear map.
\end{proof}

\begin{definition}
For any $\lambda\in \Gamma_{a,L}$ we call $\phi_{a,L}^\lambda$ a {\it type II deformation associated with the vertex $a$}.
\end{definition}

\begin{remark} \label{rmk-typeII}
Suppose $\Nb(a)$ is disconnected.
Let $L$ be a proper subgraph of $\Nb(a)$ which is a union of connected components of $\Nb(a)$.
Let $B$ be the vertex set of $L$ and $A=\V(\Nb(a)) \backslash B$.
Then $\Gamma(L)=\emptyset$ and $\phi_{a,L}^1$ is the homomorphism corresponding to the separation at vertex $a$ and separation pair $(A,B)$.

Suppose $G$ has a loop on vertex $a$.
Let $L_1=\{a\}$.
The induced subgraph of $\Nb(a)$ on $L_1$ is a connected component of $\Nb(a)$ and the homomorphism $\phi^1_{a,L_1}$ corresponds to separation at vertex $a$ and separation pair $(V(\Nb(a))\backslash \{a\},\{a\})$.
\begin{enumerate}
\item If $N(a)=\{a\}$ then the loop on $a$ is an isolated loop and $\phi^1_{a,L_1} = \frac{1}{2} \frac{\partial}{\partial a}$. In this case the only separation at $a$ is a trivial deformation.
\item If $N(a) \neq \{a\}$ then the separation $\phi^1_{a,L_1}$ is a nontrivial deformation. Suppose $\phi^1_{a,L_1} = \sum_{g\in V(G)} r_g \frac{\partial}{\partial g}$ is trivial. Since $\phi^1_{a,L_1}(a^2) = 2r_a a = a$, $r_a=\frac{1}{2}$. Now for $x\in N(a)$ distinct from $a$ we have $\phi^1_{a,L_1}(ax) = \frac{1}{2} x + r_x a =0$ which is a contradiction. This implies that $\phi^1_{a,L_1}$ corresponds to a nontrivial deformation.
If we choose $L_2 = V(\Nb(a))$ then $\Gamma(L_2)=\emptyset$ and $\phi_{a,L_1}^1 + \phi^1_{a,L_2} = \frac{\partial}{\partial a}$. Therefore the homomorphism $\phi^1_{a,L_2}$ also corresponds to a nontrivial deformation.
\end{enumerate}

Note that if $G$ does not have a loop on vertex $a$ and we choose $L$ equal to $V(\Nb(a))$ then $\phi^1_{a,L}= \frac{\partial}{\partial a}$.
\end{remark}



\begin{lemma} \label{TypeIInontrivial}
For a vertex $a$ and a nonempty subset $L\subseteq V(\Nb(a))$, let $\phi^{\lambda}_{a,L}$ be a nonzero type II deformation associated with vertex $a$. If $\phi^{\lambda}_{a,L}$ satisfies either of the following conditions
\begin{enumerate}
\item $G$ dos not have any loop on $a$ and $\phi^{\lambda}_{a,L} \neq \lambda \frac{\partial}{\partial a}$;
\item $G$ has a loop on $a$ and $\phi^{\lambda}_{a,L}$ is not equal to $\lambda \phi^1_{a,\{a\}}$ nor $\lambda \phi^1_{a,V(\Nb(a))}$.
\end{enumerate}
then $\phi^{\lambda}_{a,L}$ corresponds to a nontrivial deformation.
\end{lemma}

\begin{proof}
Suppose $\phi$ corresponds to a trivial first order deformation and $\phi = \sum_{g \in V(G)} r_g \frac{\partial}{\partial g}$.
First consider the case where there is no loop on $a$.
In this case since $\phi$ is nonzero there exists some $x\in L$ such that $x\lambda$ is nonzero in $R/I$. For such vertex $x$ we have $\phi(ax) = r_ax+r_xa = \lambda x$ and since $a$ does not divide $\lambda$,
$r_a = \lambda + r$ such that $rx + r_x a = 0$ in $R/I$. Thus 
\[\phi = (\lambda+r) \frac{\partial}{\partial a} + \sum_{\substack{g \in V(G)\\g\neq a}} r_g \frac{\partial}{\partial g}.\]
Since $\phi \neq \lambda\frac{\partial}{\partial a}$, there exists some $y\in \Nb(a) - (L\cup \Gamma)$ such that $\lambda y$ is not in $I$. We have
\[
\phi(ay) = (\lambda+r) y + \sum_{\substack{g \in V(G)\\g\neq a}} r_g \frac{\partial}{\partial g} (ay) = \lambda y + r y + r_y a = 0.
\]
The term $\lambda y$ can not be canceled with a term in $r y$. It follows that it cancels with a term in $r_y a$ and $a|\lambda$ which is a contradiction. Therefore $\phi$ can not be written as a combination of derivations in this case.

Suppose there is a loop on $a$. Let $\psi_1$ be the homomorphism $\phi^1_{a,\{a\}}$ and let $\psi_2 = \phi^1_{a,V(\Nb(a))}$.
Since $\phi\neq \lambda \psi_1$, there is a vertex $x\in L$ distinct from $a$ such that $\lambda x\notin I$.
Furthermore from our assumption that $\phi\neq \lambda \psi_2$ it follows that there exists a vertex $y \in V(\Nb(a))\backslash (L\cup \Gamma)$ such that $\lambda y\notin I$.
Now a similar argument to the previous case shows that $\phi$ corresponds to a nontrivial deformation.
\end{proof}

\begin{theorem}
\label{thm-Hom-gens}
As $ab$ varies in the set of edges of $G$ and $a$ varies in the set of vertices of $G$, the homomorphisms $\phi_{ab}^\lambda$ for $\lambda\in\Lambda_{ab}$ alongside with the homomorphisms $\phi_{a,L}^\lambda$ for nonempty $L\subseteq V(\Nb(a))$ and $\lambda\in\Gamma_{a,L}$ define a generating set for $\HomI$.
\end{theorem}

\begin{lemma} \label{lem-02}
Let $I$ be the edge ideal of a graph $G$.
Let $\phi \in \Hom_R(I,R/I)$ be a homomorphism and let $ab$ be an edge of $G$.
Suppose $\phi(ab)$ is written as a linear combination of monomials in $R$ and let $m$ be a monomial in $\phi(ab)$. Then either $\gcd(m,ab) \neq 1$ or $m$ is divisible by a monomial in $\Lambda_{ab}$.
\end{lemma}

\begin{proof}
Suppose $\gcd(m,ab)=1$.
If $ab$ is an isolated edge then $1|m$ and there is nothing to prove. 
In the remaining of the proof suppose $ab$ is not an isolated edge.

Let $r\in \kk$ be the coefficient of $m$ in $\phi(ab)$.
Without loss of generality if a vertex $c$ is adjacent to $b$ then the relation $c(ab)-a(bc)$ implies $c\phi(ab)-a\phi(bc)=0$. If $rcm$ is canceled out by a monomial in $a\phi(bc)$ then $a|m$ which is a contradiction. Therefore $rcm$ is in $I$. This means that there exist a vertex $c'$ adjacent to $c$ such that $c'|m$. Hence for any vertex $c$ adjacent to $a$ or $b$, $m$ contains a vertex adjacent to $c$. This exactly means that $m$ is divisible by a monomial in $\Lambda_{ab}$.
\end{proof}

\begin{proof}[Proof of Theorem \ref{thm-Hom-gens}]
Let $N$ be the submodule of $\HomI$ generated by all of these homomorphisms.
Suppose $ab$ is an edge of $G$ and $\phi(ab)$ contains a term $rm$ in which $m$ is a monomial in $R$ and $r$ is a scalar. We show that modulo $N$ we can eliminate this term.

If $\gcd(m,ab)=1$ then by \ref{lem-02}, there is some $\lambda\in\Lambda_{ab}$ that divides $m$. Suppose $m=\lambda k$ then modulo $N$, $\phi = \phi- rk \phi_{ab}^\lambda$ and $(\phi- rk \phi_{ab}^\lambda)(ab)$ does not have the term $rm$.
Note that for any generator $xy\in I$, $\Mon((\phi-rk\phi_{ab}^\lambda)(xy)) \subseteq \Mon(\phi(xy))$.

Now without loss of generality suppose $b|m$ and $m=bm'$.
Let $L$ be the set of all vertices $x$ in $N(a)$ such that the monomial $xm'$ appears with a nonzero coefficient in $\phi(ax)$ and also $xm'\notin I$.
Choose an element $x'$ in $\Gamma(L)$. By definition of $\Gamma(L)$, there is an element $x\in L$ such that $xx'\notin I$. The relation $x'(ax)-x(ax')$ implies that $x'\phi(ax) - x\phi(ax')=0$ in $R/I$.
Since $\phi(ax')$ does not have the monomial $x'm'$, the monomial $x'xm'$ can not be canceled out. Hence $x'xm'$ is in $I$ which implies that $x'm' \in I$. Therefore for some vertex $\lambda_{x'}\in\Gamma_{x'}$ we have $\lambda_{x'}|m'$.
Now let $\lambda$ be the least common multiple of $\lambda_{x'}$ for all $x'\in \Gamma(L)$. It follows from the definition of $\Gamma_{a,L}$ that $\lambda \in \Gamma_{a,L}$.
Suppose $m'=k\lambda$.
Now modulo $N$, $\phi = \phi - rk\phi_{a,L}^\lambda$ and $(\phi - rk\phi_{a,L}^\lambda)(ab)$ does not have the term $rm$.
Furthermore, for any generator $xy\in I$, $\Mon((\phi-rk\phi^\lambda_{a,L})(xy)) \subseteq \Mon(\phi(xy))$. This means that modulo $N$ we can reduce any homomorphism in $\HomI$ to zero.
\end{proof}



\begin{example}
Let $G=C_n$ be the cycle with $n$ vertices. Let $V(C_n)=\{a_0,\ldots,a_{n-1}\}$.
If $n=3$ then there is no nonzero deformation of type I.
Choose the vertex $a_0\in V(G)$.
The induced subgraph $\Nb(a_0)$ is a graph with two isolated vertices $a_1$ and $a_2$.
For $L=\{a_1\},\{a_2\}$ or $\{a_1,a_2\}$ we have $\Gamma_{a,L} = \{1\}$.
Therefore we get 2 nontrivial deformations $\phi_{a_0,\{a_1\}}^1$ and $\phi_{a_0,\{a_2\}}^1$ at $a_0$. Similarly we have 4 nontrivial deformations at $a_1$ and $a_2$.
These 6 deformations generate $Hom_R(I(C_3),R/I(C_3))$.

For $n\geq 4$, Lemma \ref{conv-01} shows that $C_n$ does not have any deformations of type II.
Let $\mathbb{Z}_n=\{0,\ldots,n-1\}$ be the cyclic group of order $n$.
If $n=4$ or $6$ then for any $\overline{i}\in \mathbb{Z}_n$ any element of $\Lambda_{a_{\overline{i}}a_{\overline{i+1}}}$ belongs to $I$. Hence $C_4$ and $C_6$ are algebraically rigid.

$C_5$ has the following 5 type I nontrivial deformations.
\begin{align*}
a_0 a_1 & \mapsto a_3, \qquad a_1 a_2 \mapsto a_4,\\
a_2 a_3 & \mapsto a_0, \qquad a_3 a_4 \mapsto a_1,\\
a_0 a_4 & \mapsto a_2
\end{align*}
These five nontrivial deformations together with image of five derivations $\frac{\partial}{\partial a_0},\ldots,\frac{\partial}{\partial a_4}$ under $\delta^\ast$, generate $Hom_R(I(C_5),R/I(C_5))$.

Suppose $n\geq 7$.
For $\overline{i}\in \mathbb{Z}_n$, $\Lambda_{a_{\overline{i}} a_{\overline{i+1}}} = \{a_{\overline{i-2}}a_{\overline{i+3}}\}$.
We have $n$ nontrivial type I deformations defined as
\[
a_{\overline{i}} a_{\overline{i+1}} \mapsto a_{\overline{i-2}}a_{\overline{i+3}},
\]
for all $\overline{i}\in \mathbb{Z}_n$.

Therefore the only algebraically rigid cycles are $C_4$ and $C_6$.
\end{example}

\section{Rigidity of edge ideals of graphs}
\label{Rigidity}

Suppose $G$ is not a simple graph and $G$ has a loop on some vertex $x$. If $N(x)\neq \{x\}$,
then the separation at $x$ is a nontrivial deformation and $G$ is not algebraically rigid.
Now suppose  $N(x)=\{x\}$, i.e. the loop on $x$ is an isolated loop.
In this case the separation at $x$ is a trivial deformation but the type I deformation $\phi^1_{a^2}$ is a nontrivial deformation.
It follows that non square-free quadratic monomial ideals are never rigid. Therefore in this section we investigate the rigidity of simple graphs.


\begin{theorem} \label{thm-mainrigidity}
Let $I$ be the edge ideal of a simple graph $G$.
$I$ is rigid if and only if
\begin{enumerate}
\item for each edge $ab$ of $G$,
\[
\prod_{x\in \Lambda^{ab}} \Lambda_x \subseteq I, \text{ and}
\]
\item for each vertex $a$ of $G$ and subset $L\subseteq V(\Nb(a))$,
\[
\big(\prod_{x\in \Gamma(L)} \Gamma_x\big) \times \big(V(\Nb(a))\backslash( L \cup \Gamma(L)\big) \subseteq I. 
\]
\end{enumerate}
\end{theorem}

\begin{proof}
Suppose for an edge $ab$ the conditions in $(1)$ is satisfied then for each $\lambda\in\Lambda$, $\phi_{ab}^\lambda$ is zero. If for $a\in V(G)$,$L\subseteq V(\Nb(a))$ and $\lambda\in\Gamma_{a,L}$, the condition in $(2)$ is satisfied then $\phi_{a,L}^\lambda = \lambda \frac{\partial}{\partial a}$. Hence $I$ is rigid.

Conversely, suppose $I$ is rigid.
For each edge $ab$ of $G$, if $\prod_{x\in\Lambda} \Lambda_x$ contains a monomial $\lambda$ such that $\lambda\notin I$ then $\phi_{ab}^\lambda \neq 0$ and by Lemma \ref{lemTypeI}, it corresponds to a nontrivial deformation.
Therefore $\prod_{x\in \Lambda} \Lambda_x$ is a subset of $I$.
Now let $a$ be a vertex and $L$ be a subset of $V(\Nb(a))$. If for $\lambda\in \prod_{x\in \Gamma(L)}\Gamma_x$ and $x\in (V(\Nb(a))\backslash(L\cup\Gamma))$, $\lambda x$ does not belong to $I$ then by Lemma \ref{TypeIInontrivial}, $\phi_{a,L}^\lambda$ is a nontrivial deformation, which is a contradiction. Hence if $I$ is rigid then both of the conditions $(1)$ and $(2)$ hold.
\end{proof}

The following lemmata are useful in computations.

\begin{lemma}
\label{conv-01}
Let $a$  be a vertex of graph $G$ with no loop on it. Suppose either 
\begin{enumerate}
\item vertex $a$ does not lie on any 3-cycle, or
\item vertex $a$ belongs to a leaf,
\end{enumerate}
then the derivation $\frac{\partial}{\partial a}$ is the only deformation of type II associated with $a$.
\end{lemma}

\begin{proof}
Suppose $a$ satisfies $(1)$. Since $a$ does not lie on any 3-cycle, any two vertices in $\Nb(a)$ are connected and $\Nb(a)$ is a complete graph.
Therefore for any subset $L\subseteq V(\Nb(a))$, $V(\Nb(a)) = L\cup \Gamma(L)$ which implies that there are no nontrivial type II deformation at vertex $a$.

Now suppose $a$ belongs to a leaf.
If $a$ is the leaf vertex then the assertion follows from (1).
Otherwise, let $x$ be the leaf vertex adjacent to $a$.
For any subset $L\subseteq V(\Nb(a))$ if $x\in L$ then $V(\Nb(a))=L\cup \Gamma(L)$ and the only deformation at $a$ is a multiple of $\frac{\partial}{\partial a}$. If $x\notin L$ then $x\in \Gamma(L)$. Therefore for any $\lambda\in \Gamma_{a,L}$, $a$ (the only vertex adjacent to $x$) divides $\lambda$ which is a contradiction. 
\end{proof}

\begin{remark}
The first condition in \ref{thm-mainrigidity} is satisfied if and only if $G$ does not have any nonzero type I deformations and the second condition is satisfied if and only if there is no nontrivial type II deformations.
Therefore when $G$ is a simple graph with no 3-cycles then by Lemma \ref{conv-01}, $I(G)$ does not admit any nontrivial type II deformations. In this case all of the nonzero type I deformations associated with edges of $G$ form a minimal generating set for $T^1(R/I(G))$. Furthermore, the rigidity of $G$ can be checked by the simple condition that for all edge $ab$ of $G$, $\prod_{x\in \Lambda} \Lambda_x \subseteq I(G)$, where $\Lambda = N(a)\cup N(b) \backslash \{a,b\}$. Examples of simple graphs with no induced 3-cycles include the bipartite graphs and second letterplace ideals (see Example \ref{exm-letterplace} for definition).
\end{remark}

\begin{lemma}
\label{conv-02}
If an edge $ab$ is a branch then there is no nonzero type I deformation associated with edge $ab$.
\end{lemma}

\begin{proof}
This follows immediately from condition $(1)$ in Theorem \ref{thm-mainrigidity}.
\end{proof}

We conclude this section by giving another proof for the main result on rigid graphs in \cite{ABHL}. Our description for rigidity of graphs significantly simplifies the proof.

\begin{theorem}
Let $G$ be a simple graph such that $G$ does not contain any induced cycle of length 4,5 or 6. Then $G$ is rigid if and only if each edge of $G$ is a branch and each vertex of a 3-cycle of $G$ belongs to a leaf.
\end{theorem}

\begin{proof}
If each edge of $G$ is a branch then by Lemma \ref{conv-02} there is no deformation of type I.
Each vertex of $G$ either does not lie on a 3-cycle or it belongs to a leaf. It follows from Lemma \ref{conv-01} that there is no deformation of type II.
Hence $G$ is algebraically rigid.

Conversely, suppose $G$ is algebraically rigid and it does not contain any induced cycle of length 4,5 or 6.
Let $ab$ be an edge of $G$. The edge $ab$ can not be an isolated edge since otherwise $\phi_{ab}^1$ gives a nontrivial deformation.
Suppose on the contrary that $ab$ is not a branch. Let $\{x_1,\ldots,x_k\}$ be the set of vertices that are adjacent to $a$ or $b$ other than $a$ and $b$ themselves.
If $k=1$ then for any $\lambda$ in the nonempty set $N(x_1)-\{a,b\}$, $\phi_{ab}^\lambda$ is a nontrivial deformation. Hence $k\geq 2$. Now for any $\lambda\in\Lambda_{ab}$ there is $y_i$ and $y_j$ respectively adjacent to $x_i$ and $x_j$ such that $y_iy_j\in I$.
Now the induced cycle on (not necessarily distinct) vertices $a,b,x_i,y_i,x_j$ and $y_j$ contains an induced cycle of length 4,5 or 6, which is a contradiction.

Let $a$ be a vertex of a 3-cycle. Let $b_1,b_2$ be the other two vertices of this 3-cycle. If $N(a) = \{b_1,b_2\}$ then we have a separation at $a$ which is a contradiction. Let $N(a) = \{b_1,b_2,x_1,\ldots,x_k\}$. By \cite[Theorem 3.1]{ABHL} or the discussion at the end of Section \ref{Separation}, $\Nb(a)$ is connected.
Let $L=\{b_1\}$. If none of the vertices in $\Nb(a)$ is a leaf then $\Gamma_{a,L}$ is nonempty.
Therefore for any $\lambda\in \Gamma_{a,L}$, $\phi_{a,L}^\lambda$ should be a multiple of the derivation $\frac{\partial}{\partial a}$.
Choose some $\lambda\in \Gamma_{a,L}$. Since $\phi_{a,L}^\lambda$ sends $ab_2$ to zero, $b_2\lambda$ is in $I$. This implies that there is a vertex $y_i$ adjacent to some $x_i$ in $\Gamma(L)$ such that $y_i|\lambda$ and $ax_iy_ib_2$ induces a 4-cycle, which is a contradiction. Therefore each vertex of a 3-cycle belongs to a leaf.
\end{proof}



\section{Second Cotangent Cohomology}
\label{T2}
Throughout this section $G$ is a simple graph and $I(G)$ is a square-free monomial ideal.
If $G$ is not a simple graph then Lemma \ref{TwoHeadVar} which is essential to our arguments is no longer valid.

Let $I$ be an ideal in a polynomial ring $R$ and let $A=R/I$ be the quotient ring.
Let 
$$0 \to K \to R^m \to R \to A \to 0$$
be an exact sequence of $R$-modules.
We denote the submodule of $K$ generated by the Koszul relations by $K_0$.
Recall that the second cotangent cohomology $T^2(A)$ is defined as the cokernel of the induced map
$$
\Phi: \Hom_R(R^m,A) \longrightarrow \Hom_A(K/K_0,A).
$$
We fix a total order $\prec$ on $E(G)$ the edge set of $G$.
For $ab\in E(G)$, let $\epsilon_{ab}$ be the standard basis of $R^m$.
As a submodule of $R^m$, $K$ is generated by relations $r_{ab,bc}$ and $r_{ab,cd}$ defined below,
\begin{enumerate}
\item for $ab,bc\in I$ with $ab \prec bc$, $r_{ab,bc} = r_{bc,ab} = -c\epsilon_{ab}+a\epsilon_{bc}$  and,
\item for $ab,cd\in I$ with $ab \prec cd$, $r_{ab,cd} = r_{cd,ab} = -cd\epsilon_{ab}+ab\epsilon_{cd}$.
\end{enumerate}

The relations of second form are Koszul relations and they vanish in the sub-quotient $K/K_0$. Therefore any minimal generator of $K/K_0$ can be denoted by two adjacent edges $ab$ and $bc$ of $G$.
For a subset $F$ of edges of $G$ and for an edge $ab \in F$, $\sigma(F,ab)$ is defined to be the number of elements less than $ab$ in the totally ordered set $(F,\prec)$.

\begin{lemma}
Let $ab$ be a generator of $I$. The map $\phi_{ab}:K/K_0\to R/I$ defined as
\begin{enumerate}
\item for any edge $bc$ adjacent to $ab$ sending $r_{ab,bc}$ to $(-1)^{\sigma(\{ab,bc\},bc)}c$, i.e. the coefficient of $\epsilon_{ab}$ in $r_{ab,bc}$,
\item for any edge $ac$ adjacent to $ab$ sending $r_{ab,ac}$ to $(-1)^{\sigma(\{ab,ac\},ac)}c$, i.e. the coefficient of $\epsilon_{ab}$ in $r_{ab,ac}$,
\end{enumerate}
and sending any other generator of $K/K_0$ to zero, 
is an $R$-module homomorphism.
Furthermore, as $ab$ varies in $E(G)$, the homomorphisms $\phi_{ab}$ form a generating set for the image of $\Phi$.
\end{lemma}

\begin{proof}
For all edges $ab$ of $G$, let $\epsilon_{ab}$ be the standard basis of $R^m$. We also denote the $R$-module map in $\Hom_R(R^m,A)$ sending $\epsilon_{ab}$ to $1$ and the other basis elements to zero by $\epsilon_{ab}$. Note that the image of the homomorphisms $\epsilon_{ab}$ for all $ab\in E(G)$ generates the image of $\Phi$.
Now the image of the map $\epsilon_{ab}$ under $\Phi$ is exactly the map $\phi_{ab}$ defined above.
\end{proof}

\begin{example} \label{exmT2cyc3}
Let $G$ be the 3-cycle on vertex set $V(G)=\{a,b,c\}$ and let $I$ be the edge ideal of $G$ in polynomial ring $R=\kk[G]$.
The ideal $I$ has 3 relations $r_{ab,bc},r_{ac,bc}$ and $r_{ab,ac}$ which are not Koszul.
The $R$-module $K/K_0$ is generated by two elements $r_{ab,bc}$ and $r_{ab,ac}$, since $r_{ac,bc} = r_{ab,ac} - r_{ab,bc}$.
Suppose $ab\prec bc \prec ac$.
The map $K/K_0 \to R^m$ is defined as
\[
\xymatrix@C=7pc{
K/K_0 \ar^{
\begin{tiny}
\begin{bmatrix}
c &  c\\
-a & 0\\
0 &  -b
\end{bmatrix}
\end{tiny}
}[r]& R^m,
}
\]
which induces the map,
\[
\xymatrix@C=7pc{
\Hom_R(R^m,R/I) \ar^{
\begin{tiny}
\begin{bmatrix}
c &  -a & 0\\
c & 0 & -b
\end{bmatrix}
\end{tiny}
}[r]& \Hom_R(K/K_0,R/I).
}
\]
The columns of the matrix above from left to right correspond to $\phi_{ab},\phi_{bc}$ and $\phi_{ac}$. An easy computation shows that this map is surjective. Hence $T^2(R/I) = 0$.
\end{example}

\begin{remark} \label{remrel}
Let $I$ be the edge ideal of a graph $G$.
The relations of $K/K_0$ is generated by the relations of $K$ plus the generators of $K_0$.
Let $G(I)$ be the set of minimal generators of $I$ with a total order $\prec$.
For any $A\subseteq G(I)$ denote the least common multiple of monomials in $A$ by $u_A$.
For any $F=\{ab,cd,ef\} \subseteq G(I)$,
\[r_{ab,cd,ef} = 
(-1)^{\sigma(F,ab)} \frac{u_F}{u_{F\backslash\{ab\}}} r_{cd,ef} +
(-1)^{\sigma(F,cd)} \frac{u_F}{u_{F\backslash\{cd\}}} r_{ab,ef} +
(-1)^{\sigma(F,ef)} \frac{u_F}{u_{F\backslash\{ef\}}} r_{ab,cd} 
\]
generate the module of relations of $K$.
The fact that this indeed is a generating set for module of relations of $K$ follows from the exactness of the Taylor complex (see \cite[Chapter 7]{HeHiMon}).

Only the relations $r_{ab,cd,ef}$ for which at least one of $r_{ab,cd}, r_{ab,ef}, r_{cd,ef}$ is not Koszul gives a relation of $K/K_0$.
Therefore the relations of $K/K_0$ have one of the following 5 forms.
\begin{enumerate}
\item For any generator $r_{ab,bc}\in K/K_0$ with $ab\prec bc$, we have a relation
\[
b r_{ab,bc}
\]
of $K/K_0$ since $b r_{ab,bc} = -bc\epsilon_{ab} + ab \epsilon_{bc} = 0$ in $K/K_0$.
\item For $r_{ab,bc}$ and $de \in I$ with $\{a,b,c\}\cap \{d,e\} =\emptyset$, that is when we have a subgraph as
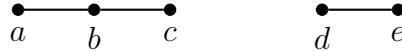
\begin{center}
\begin{tikzpicture}[scale=1, vertices/.style={draw,fill=black, circle, inner 
sep=1.5pt}]
\node [vertices, label=below:{$a$}] (0) at (0,0){};
\node [vertices, label=below:{$b$}] (1) at (1,0){};
\node [vertices, label=below:{$c$}] (2) at (2,0){};
\node [vertices, label=below:{$d$}] (3) at (4,0){};
\node [vertices, label=below:{$e$}] (4) at (5,0){};

\foreach \to/\from in {0/1, 1/2, 3/4}
\draw [-,thick] (\to)--(\from);
\end{tikzpicture}
\end{center}
we have the following relation of $K/K_0$.
\[
de r_{ab,bc},
\]
since $ (-1)^{\sigma(F,de)+1} de r_{ab,bc} =
(-1)^{\sigma(F,bc)} cr_{ab,de} +
(-1)^{\sigma(F,ab)} ar_{bc,de} = 0 \text{ in } K/K_0$.

\item For two generators $r_{ab,bc},r_{bc,cd}$, that is when we have a subgraph as
\medskip
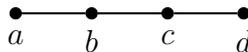
\begin{center}
\begin{tikzpicture}[scale=1, vertices/.style={draw,fill=black, circle, inner 
sep=1.5pt}]
\node [vertices, label=below:{$a$}] (0) at (0,0){};
\node [vertices, label=below:{$b$}] (1) at (1,0){};
\node [vertices, label=below:{$c$}] (2) at (2,0){};
\node [vertices, label=below:{$d$}] (3) at (3,0){};
\foreach \to/\from in {0/1, 1/2, 2/3}
\draw [-,thick] (\to)--(\from);
\end{tikzpicture}
\end{center}
we get the relation
\[
(-1)^{\sigma(F,ab)} ar_{bc,cd} + (-1)^{\sigma(F,cd)} d r_{ab,bc},
\]
since $(-1)^{\sigma(F,ab)} ar_{bc,cd} + (-1)^{\sigma(F,cd)} d r_{ab,bc} = 
 (-1)^{\sigma(F,bc)+1} r_{ab,cd} = 0$ in $K/K_0$.

\item For generators $r_{ab,ac},r_{ac,ad},r_{ad,ab}$ of $K/K_0$, that is when we have a subgraph as 
\medskip
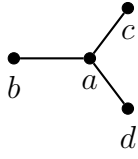
\begin{center}
\begin{tikzpicture}[scale=1, vertices/.style={draw,fill=black, circle, inner 
sep=1.5pt}]
\node [vertices, label=below:{$a$}] (0) at (0,0){};
\node [vertices, label=below:{$b$}] (1) at (-1,0){};
\node [vertices, label=below:{$c$}] (2) at (.5,.666){};
\node [vertices, label=below:{$d$}] (3) at (.5,-.666){};
\foreach \to/\from in {0/1, 0/2, 0/3}
\draw [-,thick] (\to)--(\from);
\end{tikzpicture}
\end{center}
we get the relation
\[
(-1)^{\sigma(F,ab)} br_{ac,ad} + (-1)^{\sigma(F,ac)} cr_{ab,ad}
+ (-1)^{\sigma(F,ad)} dr_{ab,ac}
\]
of $K/K_0$.
\item For $r_{ab,bc},r_{ac,bc}$ and $r_{ab,ac}$, that is when we have a subgraph as
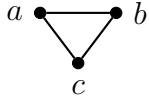
\begin{center}
\begin{tikzpicture}[scale=1, vertices/.style={draw,fill=black, circle, inner 
sep=1.5pt}]
\node [vertices, label=left:{$a$}] (0) at (0,0){};
\node [vertices, label=right:{$b$}] (1) at (1,0){};
\node [vertices, label=below:{$c$}] (2) at (.5,-.666){};
\foreach \to/\from in {0/1, 1/2, 0/2}
\draw [-,thick] (\to)--(\from);
\end{tikzpicture}
\end{center}
we get the following relation of $K/K_0$.
\[
(-1)^{\sigma(F,ab)} r_{ac,bc} + (-1)^{\sigma(F,bc)} r_{ab,ac} +
(-1)^{\sigma(F,ac)} r_{ab,bc}. 
\]
\end{enumerate}
These relations generate all the relations of $K/K_0$ and we call them relations of type (1) to (5) respectively.
\end{remark}

\begin{definition} \label{defGensHomK}
Let $I$ be the edge ideal of a graph $G$ and let $ab$ be an edge of $G$.
Let $L_a$ (resp. $L_b$) be a subset of $N(a)\backslash\{b\}$ (resp. $N(b)\backslash\{a\}$) and $\overline{L_a}$ (resp. $\overline{L_b}$) be its complement.
We shall choose $L_a$ and $L_b$ such that for any vertex $z\in N(a)\cap N(b)$ we have $z\in L_a$ if and only if $z\in L_b$. 
We define
\[
\Delta^a = \{x\in \overline{L_a}~ |~ \exists ~ y\in L_b \text{ s.t. } xy \notin I \text{ or } \exists y \in L_a \text{ s.t. } xy \notin I\}
\]
and similarly 
\[
\Delta^b = \{x\in \overline{L_b}~ |~ \exists ~ y\in L_a \text{ s.t. } xy \notin I \text{ or } \exists y \in L_b \text{ s.t. } xy \notin I\}.
\]
Let $\Delta = \Delta^a \cup \Delta^b$.
We define homomorphisms in $\HomK$ without making any further choices.

For any $x\in \Delta$ let $\Delta_x$ to be the set $N(x)\backslash\{a,b\}$.
Now define 
\[
\Delta_{L_a,L_b} = \{\sqrt{m} ~|~ m\in \prod_{x\in \Delta} \Delta_x\}.
\]

The generators of $K/K_0$ are in degree 3. Now for any $\lambda \in \Delta_{L_a,L_b}$ define a $\kk$-linear map
\[
\phi^\lambda_{L_a,L_b} : (K/K_0)_3 \longrightarrow R/I
\]
by
\[
\phi^\lambda_{L_a,L_b}(r_{e,e'}) = 
\begin{cases}
(-1)^{\sigma(\{ab,ax\},ax)} \lambda x & e=ab, e'=ax \text{ and } x\in L_a\\
(-1)^{\sigma(\{ab,bx\},bx)} \lambda x & e=ab, e'=bx \text{ and } x\in L_b\\
0 & \text{otherwise}.
\end{cases}
\]
\end{definition}

\begin{lemma}
For $L_a$ and $L_b$ as above and for any $\lambda\in \Delta_{L_a,L_b}$, $\phi^\lambda_{L_a,L_b}$ algebraically extends to a well-defined homomorphism in $\HomK$.
\end{lemma}

\begin{proof}
Let $\phi=\phi^\lambda_{L_a,L_b}$.
The only generators of $K/K_0$ that are mapped to something possibly nonzero are the generators that involve the edge $ab$.
We show that $\phi$ satisfies all the relations involving such generators.
The type (1) and (2) relations are obviously satisfied.
Now without loss of generality consider a generator $r_{ab,ax}$ for $x\in L_a$.
\begin{itemize}
\item Type (3) relations. Let $F=\{ab,ax,xy\}$. A relation $(-1)^{\sigma(F,ab)}b r_{ax,xy} + (-1)^{\sigma(F,xy)}y r_{ab,ax}$ implies $(-1)^{\sigma(F,ab)} b \phi(r_{ax,xy}) +(-1)^{\sigma(F,xy)} y \phi(r_{ab,ax}) = \pm y (\lambda x) = 0$. Now let $F=\{ax,ab,by\}$. For a relation of form $(-1)^{\sigma(F,by)} y r_{ab,ax} + (-1)^{\sigma(F,ax)} x r_{ab,by}$ there are 3 possibilities. If $y\in L_b$ then $(-1)^{\sigma(F,by)}y \phi(r_{ab,ax}) + (-1)^{\sigma(F,ax)} x \phi(r_{ab,by}) =
(-1)^{\sigma(F,by)+\sigma(\{ab,ax\},ax)} y(\lambda x) +
(-1)^{\sigma(F,ax)+\sigma(\{ab,by\},by)}  x(\lambda y) =0$.
It is not hard to show that
$\sigma(F,by)+\sigma(\{ab,ax\},ax)$ and $\sigma(F,ax)+\sigma(\{ab,by\},by)$ have opposite parities.
If $y\in \Delta^b$ then  $(-1)^{\sigma(F,by)}y \phi(r_{ab,ax}) + (-1)^{\sigma(F,ax)} x \phi(r_{ab,by}) =
\pm y(\lambda x) =0$ since there is a vertex $z$ adjacent to $y$ such that $z|\lambda$.
Suppose $y$ does not lie in $L_b \cup \Delta^b$. By definition of $\Delta^b$, $y$ is adjacent to $x$, hence $(-1)^{\sigma(F,by)}y \phi(r_{ab,ax}) + (-1)^{\sigma(F,ax)} x \phi(r_{ab,by}) =\pm y(\lambda x) =0$.
\item Type (4) relations. Let $F=\{ab,ax,ay\}$. Consider a relation
\[(-1)^{\sigma(F,ay)} y r_{ab,ax} +(-1)^{\sigma(F,ax)} x r_{ab,ay} + (-1)^{\sigma(F,ab)} b r_{ax,ay}.\]
We have 
\begin{align*}
& (-1)^{\sigma(F,ay)} y \phi(r_{ab,ax}) +(-1)^{\sigma(F,ax)} x \phi(r_{ab,ay}) +  (-1)^{\sigma(F,ab)} b \phi(r_{ax,ay}) =\\
& (-1)^{\sigma(F,ay)+\sigma(\{ab,ax\},ax)} y(\lambda x) + (-1)^{\sigma(F,ax)} x \phi(r_{ab,ay}).
\end{align*}
If $y\in L_a$ then obviously $\phi$ satisfies the relation since 
$\sigma(F,ay)+\sigma(\{ab,ax\},ax)$ and $\sigma(F,ax)+\sigma(\{ab,ay\},ay)$ has opposite parities. If $y\in \Delta^{a}$ then $(-1)^{\sigma(F,ay)+\sigma(\{ab,ax\},ax)} y(\lambda x) + (-1)^{\sigma(F,ax)} x \phi(r_{ab,ay}) = (-1)^{\sigma(F,ay)+\sigma(\{ab,ax\},ax)} y(\lambda x) =0$ since $y\lambda$ lies in $I$.
If $y \notin L_a \cup \Delta^a$ then $y$ is adjacent to $x$ and $\phi$ also satisfies the relation in this case.
\item Type (5) relations. Let $F=\{ab,ax,bx\}$. A relation $(-1)^{\sigma(F,bx)} r_{ab,ax} +(-1)^{\sigma(F,ab)} r_{ax,bx} + (-1)^{\sigma(F,ax)} r_{ab,bx}$ implies $(-1)^{\sigma(F,bx)} \phi(r_{ab,ax}) +(-1)^{\sigma(F,ab)} \phi(r_{ax,bx}) + (-1)^{\sigma(F,ax)} \phi(r_{ab,bx}) =
(-1)^{\sigma(F,bx)+\sigma(\{ab,ax\},ax)} x\phi(r_{ab,ax}) + (-1)^{\sigma(F,ax)+\sigma(\{ab,bx\},bx)} x\phi(r_{ab,bx}) =0$ since $x \in L_a$ if and only if $x\in L_b$ and
$\sigma(F,bx)+\sigma(\{ab,ax\},ax)$ and $\sigma(F,ax)+\sigma(\{ab,bx\},bx)$ have opposite parities.
\end{itemize}
For generators $r_{ab,bx}$ when $x$ is in $L_b$, proof is similar.
\end{proof}

The following lemma shows when the homomorphisms $\phi^\lambda_{L_a,L_b}$ for $\lambda\in \Delta_{L_a,L_b}$ does not lie in the image of $\Phi$.

\begin{lemma} \label{rem-nonvangensT2}
For an edge $ab$ of $G$, let $L_a$ and $L_b$ be chosen as above.
Let $\lambda$ be a monomial in $\Delta_{L_a,L_b}$.
If $\phi^\lambda_{L_a,L_b}$ is nonzero and it is not equal to the homomorphism $\lambda \phi_{ab}$ then $\phi^\lambda_{L_a,L_b}$ is a nonzero element in $T^2(R/I(G))$.
\end{lemma}

\begin{proof}
Suppose $\phi^\lambda_{L_a,L_b} = \sum_{xy\in I(G)} r_{xy}\phi_{xy}$.
At least one of $L_a$ or $L_b$ is nonempty.
Now without loss of generality suppose $L_a\neq \emptyset$, then for $x\in L_a$ for which $\lambda x\notin I$ we have
\[
\phi^\lambda_{L_a,L_b}(r_{ab,ax}) = (-1)^{\sigma(\{ab,ax\},ax)} r_{ab} x + (-1)^{\sigma(\{ab,ax\},ab)} r_{ax} b = (-1)^{\sigma(\{ab,ax\},ax)} \lambda x.\]
Since $b \nmid \lambda$, $r_{ab} = \lambda + r$ such that
$(-1)^{\sigma(\{ab,ax\},ax)} r x+ (-1)^{\sigma(\{ab,ax\},ab)} r_{ax} b=0$.

Since $\phi^\lambda_{L_a,L_b} \neq \lambda \phi_{ab}$, there exists either an element $y\in \overline{L_a}$ with $\lambda y \notin I$ such that $\phi^\lambda_{L_a,L_b}(r_{ab,ay}) \neq \lambda \phi_{ab} (r_{ab,ay})$ or an element $y\in \overline{L_b}$ with $\lambda y \notin I$ such that $\phi^\lambda_{L_a,L_b}(r_{ab,by}) \neq \lambda \phi_{ab} (r_{ab,by})$.
If $y\in \overline{L_a}$ then
\[
\phi^\lambda_{L_a,L_b}(r_{ab,ay}) = (-1)^{\sigma(\{ab,ay\},ay)} (\lambda+r) y + (-1)^{\sigma(\{ab,ay\},ab)} r_{ay} a =0,
\]
in $R/I$.
It follows that $a|\lambda$ which is in contradiction with definition of $\Delta_{L_a,L_b}$.
If $y \in \overline{L_b}$ then
\[
\phi^\lambda_{L_a,L_b}(r_{ab,by}) = (-1)^{\sigma(\{ab,by\},by)} (\lambda+r) y + (-1)^{\sigma(\{ab,by\},ab)} r_{by} b =0,
\]
in $R/I$. Similar to above this also results in a contradiction.
\end{proof}

Before giving the main result of this section in Theorem \ref{thmMainT2}, we develop some lemmata.

\begin{lemma}
\label{TwoHeadVar}
Let $\phi \in Hom_R(K/K_0,R/I)$. For a generator $r_{ab,bc}$ of $K/K_0$ suppose $\phi(r_{ab,bc})$ is written as a linear combination of monomials in $R$. If $m$ is such monomial then either $a|m$ or $c|m$.
\end{lemma}

\begin{proof}
The relation $br_{ab,bc}$ implies that $b\phi(r_{ab,bc}) =0$ in $R/I$. Therefore $m$ contains a vertex $d$ adjacent to $b$. If $d$ is distinct from $a$ and $c$, then we have the type (4) relation
\[
(-1)^{\sigma(F,ab)}ar_{bc,bd} +(-1)^{\sigma(F,bc)} c r_{bd,ab} +(-1)^{\sigma(F,bd)} d r_{ab,bc}=0
\]
If $dm$ cancels out by a term in $(-1)^{\sigma(F,ab)}a\phi(r_{bc,bd}) +(-1)^{\sigma(F,bc)} c \phi(r_{bd,ab})$ then either $a$ or $c$ divides $m$. Otherwise $dm$ is in $I$ and this implies that $m$ also contains a vertex adjacent to $d$ which is a contradiction since we assumed that $m$ is nonzero in $R/I$.
\end{proof}

\begin{lemma} \label{lemVan-ac}
Let $\phi$ be a homomorphism in $\HomK$.
Let $r_{ab,bc}$ be a generator of $K/K_0$ such that $ab$ does not lie on any 3-cycle.
If $\phi(r_{ab,bc})$ contains a term of form $racm$ for a monomial $m\in R$ and scalar $r\in \kk$ then modulo $\im \Phi$ we can eliminate such term.
More precisely, there is a homomorphism $\psi$ such that modulo $\im \Phi$, $\psi=\phi$ and for all generators $r_{e,e'}$ of $K/K_0$, $\Mon(\psi(r_{e,e'})) \subseteq \Mon(\phi(r_{e,e'}))$ and $\psi(r_{ab,bc})$ does not have the term $racm$.
\end{lemma}

\begin{proof}
Consider a nonzero term $racm$ in which $m$ is a monomial and $r$ is a scalar.
Let $L$ be the set of vertices $x$ in $N(b)\backslash \{a\}$ such that $axm \notin I$ and $axm$ appears with a nonzero coefficient in $\phi(r_{ab,bx})$ and let $\overline{L}$ be its complement in $N(b)\backslash\{a\}$.
Let $\psi = r am \phi_{ab} + \sum_{x\in \overline{L}} rxm \phi_{bx}$ if $bc \prec ab$ and let $\psi = - r am \phi_{ab} - \sum_{x\in \overline{L}} rxm \phi_{bx}$ if $ab\prec bc$.
Note that $\psi(r_{ab,bc}) = racm$.
Therefore $\phi-\psi$ eliminates the term $racm$.
We show that for any generator $r_{e,e'}$ of $K/K_0$, either $\psi(r_{e,e'})$ is zero or $\phi(r_{e,e'})$ contains the same monomials as $\psi(r_{e,e'})$ but possibly with different coefficients.
That is for any two adjacent edges $e,e'$ and generator $r_{e,e'}$ of $K/K_0$
\[
\Mon((\phi-\psi)(r_{e,e'})) \subseteq \Mon(\phi(r_{e,e'})).
\]

The only generators of $K/K_0$ that are mapped to something nonzero are the generators that contain $ab$ or $bx$ for $x \in \overline{L}$.
For any $x\in L$, $\psi(r_{ab,bx}) = \pm raxm$ and $\phi$ has a nonzero term $r'axm$ by definition of $L$.
Therefore in $(\phi - \psi)(r_{ab,bx})$ either the monomial $axm$ vanishes or it remains with a different coefficient.
For $x\in \overline{L}$, we have $\psi(r_{ab,bx})= \pm((-1)^{\sigma(\{ab,bx\},bx)} raxm + (-1)^{\sigma(\{ab,bx\},ab)} raxm) = 0$.
For a generator of the form $r_{ab,ay}$ we have $\psi(r_{ab,ay}) = \pm ram(y) =0$ in $R/I$.
Also for a generator $r_{bx,xy}$ with $x\in \overline{L}$ we have $\psi(r_{bx,xy}) = \pm rxym=0$ in $R/I$.
Now consider a generator of form $r_{bx,by}$.
If $x,y$ are in $\overline{L}$ then $\psi(r_{bx,by}) = \pm ( (-1)^{\sigma(\{bx,by\},by)} x(ym) + (-1)^{\sigma(\{bx,by\},bx)} y(xm)) =0$.
If $x\in \overline{L}$ and $y\in L$ then $\psi(r_{bx,by}) = \pm( (-1)^{\sigma(\{bx,by\},by)} rxym)$.
Suppose $xym \notin I$ then for $F=\{ab,bx,by\}$ consider the type (4) relation below
\[
(-1)^{\sigma(F,ab)} a \phi(r_{bx,by}) + (-1)^{\sigma(F,bx)} x \phi(r_{ab,by}) + (-1)^{\sigma(F,by)} y \phi(r_{ab,bx}).
\]
Since $ab$ does not lie on any 3 cycles $axy\notin I$. Hence $axym \notin I$.
Suppose $\phi(r_{ab,by})$ has the term $\pm r'aym$.
If the monomial $xym$ does not appear in $\phi(r_{bx,by})$ with a nonzero coefficient then $r'axym$ should cancel with a monomial in $y \phi(r_{ab,bx})$.
Hence $\phi(r_{ab,bx})$ has the nonzero term $\pm r'axm$ which is against our assumption that $x\notin L$.
This completes the proof.
\end{proof}

\begin{theorem} \label{thmMainT2}
Let $G$ be a graph with no 3-cycles and let $I$ be its edge ideal.
The second cotangent cohomology module $T^2(R/I)$ vanishes if and only if for any edge $ab$ of $G$ and any $L_a$ and $L_b$ as in \ref{defGensHomK} we have
\begin{equation} \label{eqT2van}
\prod_{x\in \Delta} \Delta_x \times ((N(a)\cup N(b)) \backslash (\{a,b\} \cup L_a \cup L_b \cup \Delta) \subseteq I.
\end{equation}
\end{theorem}

\begin{proof}
Firstly we show that as $ab$ varies in $E(G)$ the homomorphisms $\phi^\lambda_{L_a,L_b}$ form  a generating set for $T^2(R/I)$.
Let $N$ be a submodule of $\HomK$ generated by these homomorphisms.
Suppose for a generator $r_{ab,bc}$ with $ab\prec bc$ of $K/K_0$, $\phi(r_{ab,bc})$ contains a term of form $racm$ in which $m$ is a monomial and $r$ is a scalar. Then by Lemma \ref{lemVan-ac} we can eliminate such term modulo $\im \Phi$.
Now suppose $\phi(r_{ab,bc})$ contains a nonzero term of form $rcm$ for a monomial $m$ such that $a \nmid m$ and a scalar $r$.
Let
\[
L_a = \{x\in N(a)\backslash\{b\} ~|~ \phi(r_{ab,ax}) \text{ has the term $(-1)^{\sigma(\{ab,ax\},ax)+1} rxm$ and } xm\notin I\}
\]
and
\[
L_b = \{x\in N(b)\backslash\{a\} ~|~ \phi(r_{ab,bx}) \text{ has the term $(-1)^{\sigma(\{ab,bx\},bx)+1} rxm$ and } xm\notin I\}.
\]
We show that for some $\lambda\in\Delta_{L_a,L_b}$ and a monomial $k$, $(\phi - rk \phi^\lambda_{L_a,L_b})(r_{ab,bc})$ does not have the term $rcm$. Furthermore we show that for any generator $r_{e,e'}$ of $K/K_0$ we have
\[
\Mon((\phi - rk \phi^\lambda_{L_a,L_b})(r_{e,e'})) \subseteq
\Mon(\phi(r_{e,e'})) 
\]
By definition of $L_a$ for any $x\in L_a$, $\phi(r_{ab,ax})$ has the term $\pm rxm$. The same argument holds for $L_b$.
Therefore we only need to show that $\lambda | m$.
Choose a vertex $x$ in $\Delta$. First let $x\in \Delta^a$.
Suppose there is a vertex $y \in L_b$ such that $xy \notin I$.
For $F=\{ab,ax,by\}$ consider the type (3) relation
\[
(-1)^{\sigma(F,ax)} x\phi(r_{ab,by}) + (-1)^{\sigma(F,by)} y \phi(r_{ab,ax}).
\]
Since $ym$ and $xy$ does not belong to $I$, $xym \notin I$.
The term $(-1)^{\sigma(F,ax)+\sigma(\{ab,by\},by)+1} x(rym)$ either belongs to $I$ or it cancels with a term in $y\phi(r_{ab,ax})$. If it cancels out then $\phi(r_{ab,ax})$ has the term $(-1)^{\sigma(F,ax)+\sigma(\{ab,by\},by)+\sigma(F,by)} rxm$. The parity of $\sigma(F,ax)+\sigma(\{ab,by\},by)$ is the same as the parity of $\sigma(\{ab,ax\},ax)+1$. Thus $\phi(r_{ab,ax})$ has the term $(-1)^{\{ab,ax\},ax)+1} rxm$ which is against our assumption that $x \notin L_a$.
Now $xym \in I$ implies $xm\in I$ and this means that there is a vertex $z$ in $\Delta_x$ such that $z |m$.

Now suppose there is a vertex $y\in L_a$ such that $xy\notin I$.
Let $F=\{ab,ax,ay\}$. The type (4) relation $r_{ab,ax,ay}$ implies
\[
(-1)^{\sigma(F,ab)} b \phi(r_{ax,ay}) + (-1)^{\sigma(F,ax)} x\phi(r_{ab,ay})+
(-1)^{\sigma(F,ay)} y\phi(r_{ab,ax})=0.
\]
If $x(r'ym) \notin I$ and it cancels out by sum of a term in $(-1)^{\sigma(F,ab)} b \phi(r_{ax,ay})$ and a term in $(-1)^{\sigma(F,ay)} y\phi(r_{ab,ax})$ then $b|m$ which is a contradiction since we assumed that $cm$ is nonzero in $R/I$.
On the other hand it can not be canceled by a term in $(-1)^{\sigma(F,ay)} y\phi(r_{ab,ax})$ since $x\notin L_a$.
Thus $xym\in I$ and there is a vertex in $\Delta_x$ that divides $m$.

In case $x\in \Delta^b$ and there is a vertex $y\in L_a$ such that $xy\notin I$, the argument is similar to above.
Now suppose for $x\in \Delta^b$ there is a vertex $y\in L_b$ such that $xy\notin I$.
For $F=\{ab,bx,by\}$ the type (4) relation $r_{ab,bx,by}$ implies
\[
(-1)^{\sigma(F,ab)} a \phi(r_{bx,by}) + (-1)^{\sigma(F,bx)} x\phi(r_{ab,by})+
(-1)^{\sigma(F,by)} y\phi(r_{ab,bx})=0.
\]
If $x(r'ym) \notin I$ then it can not be canceled by sum of a term in $(-1)^{\sigma(F,ab)} b \phi(r_{bx,by})$ and a term in $(-1)^{\sigma(F,by)} y\phi(r_{ab,bx})$. Since otherwise $a|m$ which is against our assumption.
Similarly it can not be canceled by a term in $(-1)^{\sigma(F,by)} y\phi(r_{ab,bx})$. Hence also in this case $xym \in I$ and some vertex in $\Delta_x$ divides $m$.

Therefore modulo $N+ \im \Phi$ we can reduce $\phi$ to zero. This completes the proof that the homomorphisms in $N$ generate $T^2(R/I)$.

Now suppose for any edge $ab$ of $G$ the condition \eqref{eqT2van} holds.
Then $\phi^\lambda_{L_a,L_b} = \lambda \phi_{ab}$ and therefore it vanishes in $T^2(R/I)$.
Conversely if $T^2(R/I)$ vanishes then for any $L_a$ and $L_b$, either $\phi^\lambda_{L_a,L_b}$ is zero or it is equal to $\lambda \phi_{ab}$ by Lemma \ref{rem-nonvangensT2}. This means that for any $x\in (N(a)\cup N(b)) - (\{a,b\} \cup L_a \cup L_b \cup \Delta)$,
$\phi(r_{ab,bx}) = \lambda x =0$ in $R/I$. which means that $\lambda x \in I$.
This shows that \eqref{eqT2van} holds for any edge $ab$ and any $L_a$ and $L_b$.
\end{proof}

 It is worth mentioning that when $G$ does not have any 3-cycle and for an edge $ab$, both of $L_a$ and $L_b$ are nonempty then $N(a)\cup N(b) = \{a,b\} \cup L_a \cup L_b \cup \Delta$ and the condition (1) of Theorem above is satisfied. Therefore for such graphs we only need to examine condition (1) when exactly one of two sets $L_a$ and $L_b$ is nonempty.

\begin{corollary} \label{corT2}
If $G$ is a graph with no induced 3 or 4 cycles then $T^2(R/I(G))$ vanishes.
\end{corollary}

\begin{proof}
Consider an edge $ab$ in $G$. Let $L_a$ (resp. $L_b$) be a subset of $N(a)\backslash\{b\}$ (resp. $N(b)\backslash\{a\}$).
By definition of $\Delta$ it is easy to see that $N(a)\cup N(b) = \{a,b\} \cup L_a \cup L_b \cup \Delta$. Now the assertion follows from Theorem \ref{thmMainT2}.
\end{proof}

\begin{example}
Let $I$ be the edge ideal of the 4-cycle.
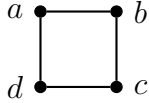
\begin{center}
\begin{tikzpicture}[scale=1, vertices/.style={draw,fill=black, circle, inner 
sep=1.5pt}]
\node [vertices, label=left:{$d$}] (3) at (0,0){};
\node [vertices, label=right:{$c$}] (2) at (1,0){};
\node [vertices, label=left:{$a$}] (0) at (0,1){};
\node [vertices, label=right:{$b$}] (1) at (1,1){};
\foreach \to/\from in {0/1, 1/2, 2/3, 3/0}
\draw [-,thick] (\to)--(\from);
\end{tikzpicture}
\end{center}
The ideal $I$ has 4 generators and as a subquotient of $S^4$, $K/K_0$ is generated by columns of the matrix
\[
\bordermatrix{
~ & r_1 & r_2 & r_3 & r_4\cr
\epsilon_{ab} & d & c & 0 & 0\cr
\epsilon_{bc} & 0 &-a & d & 0\cr
\epsilon_{cd} & 0 & 0 &-b & a\cr
\epsilon_{ad} &-b & 0 & 0 &-c
}
\]
in which $r_1=r_{ab,ad},r_2=r_{ab,bc},r_3=r_{bc,cd}$ and $r_4=r_{ad,cd}$.
Furthermore $\Hom_R(K/K_0,R/I)$ is generated by following 8 homomorphisms,
\begin{align*}
r_1 & \mapsto b, \qquad r_1 \mapsto d\\
r_2 & \mapsto a, \qquad r_2 \mapsto c\\
r_3 & \mapsto b, \qquad r_3 \mapsto d\\
r_4 & \mapsto a, \qquad r_4 \mapsto c
\end{align*}
where $r_i\mapsto x$ denotes the homomorphism sending $r_i$ to $x$ and any other generator of $K/K_0$ to zero.
This shows that $T^2(R/I)$ does not vanish.
It follows from Example \ref{exmT2cyc3} and Corollary \ref{corT2} that $C_4$ is the only cycle with a non-vanishing second cotangent cohomology module.
\end{example}

\begin{example} \label{exm-letterplace}
Let $P$ be a finite poset. The {\it second letterplace ideal} of $P$ is the quadratic monomial ideal in the polynomial ring $R=\kk[x_{[2]\times P}]$ generated by monomials $x_{1,p} x_{2,q}$ for any relation $p\leq q$ in $P$. We denote this ideal by $L(2,P)$. For simplicity we denote a variable $x_{i,p}$ in $R$ by $p_i$ for $i=1,2$.
In \cite[Corollary 2.13]{ABHL} it is shown that $L(2,P)$ is never rigid.
Here we show that $T^2(R/L(2,P))$ vanishes if and only if $\height(P)\leq 1$.
If $\height(P)\leq 1$ then the graph of $L(2,P)$ is a tree and the second cotangent cohomology module vanishes by Corollary \ref{corT2}.

Conversely let $p < q<r$ be a chain of length 2 in $P$.
Consider the generator $q_1q_2$ of $L(2,P)$.
Let $L_{q_1}=\emptyset$ and $L_{q_2}=\{p_1\}$. Since $L(2,P)$ does not have any 3-cycles we have $\Delta^{q_2} = N(q_2)\backslash \{p_1,q_1\}$.
Any element in $N(q_1)\backslash\{q_2\} = \{x_2| q < x\}$ is adjacent to $p_1$. Thus $\Delta^{q_1}=\emptyset$. Any element of $\Delta^{q_2}$ is of form $x_1$ for some $x\in P$ and $\Delta_{x_1} = \{y_2|x<y, y\neq q\}$. Note that $r_2$ is an element in $N(q_1)\cup N(q_2)$ that does not belong to
\[
\{q_1,q_2\} \cup L_{q_1} \cup L_{q_2} \cup \Delta= N(q_2)\cup \{q_2\}.
\]
Obviously, $r_2 \prod_{x\in \Delta} \Delta_x \nsubseteq I$. Hence $T^2(R/L(2,P))$ does not vanish.

Even though the second cotangent cohomology does not vanish it is shown in \cite{FN} that when the Hasse diagram of the poset $P$ is a rooted tree then all the first order deformations lift to higher order deformations. 
\end{example}

\bibliographystyle{alpha}
\bibliography{Bibliography}

\end{document}